%% file: main.tex
\documentclass[a4paper,11pt]{amsart}
\usepackage[pdftex]{color,graphicx}
\usepackage{amssymb}
\usepackage{amsfonts}
\usepackage{esint}
\usepackage{amsmath}
\usepackage{amsrefs}
\usepackage{epsfig}
\usepackage{amsthm}
\usepackage{color}
\usepackage[]{epsfig}
\usepackage[]{pstricks}
\usepackage[utf8]{inputenc}

\newtheorem{theorem}{Theorem}[section]

\newtheorem{lemma}{Lemma}[section]

\theoremstyle{definition}

\newtheorem{example}{Example}[section]

\theoremstyle{remark}
\newtheorem{remark}{Remark}[section]

\newcommand{\ran}{\rangle}
\newcommand{\lan}{\langle}

\newcommand{\interno}[2]{\left\langle #1 ,#2 \right\rangle}

\DeclareMathOperator{\trace}{trace}

\numberwithin{equation}{section}
\title[Self-Similar Solutions to IMCF of Ruled in $\mathbb{L}^3$]{Classification of ruled surfaces as homothetic self-similar solutions of the inverse mean curvature flow in the Lorentz-Minkowski 3-space}
\author{Greg\'orio Silva Neto \and Vanessa Silva}
\date{August 22, 2023}

\address{Instituto de Matem\'atica,
Universidade Federal de Alagoas,
Macei\'o, AL, 57072-900, Brazil}
\email{gregorio@im.ufal.br}

\address{Instituto Federal de Alagoas, Maceió, AL, 57020-510, Brazil}
\email{vanessa.silva@im.ufal.br}

\begin{document}
\footnotetext{G. Silva Neto was partially supported by National Council for Scientific and Technological Development (CNPq) and Alagoas Research Foundation (Fapeal) of Brazil}
\subjclass[2020]{53E10, 53C50, 53C42, 53B30}
\begin{abstract}
In this paper, we classify the nondegenerate ruled surfaces in the three-dimensional Lorentz-Minkowski space that are homothetic self-similar solutions for the inverse mean curvature flow. This classification shows the existence of two classes of non-cylindrical homothetic solitons: one with lightlike rulings and another one with non-lightlike rulings. 
\end{abstract}
\maketitle

\section{Introduction}

Let $\mathbb{L}^3$ be the Lorentz-Minkowski space, i.e., the vector space $\mathbb{R}^3$ equipped with the metric
\begin{equation}
    \langle \text{ },\text{ }\rangle=dx^2+dy^2-dz^2
\end{equation}
and $X:M\rightarrow\mathbb{L}^3$ be a smooth immersion of a nondegenerate surface $M,$ with nonzero mean curvature. We say that a one parameter family of immersions $X_t=X(\cdot,t):M\rightarrow \mathbb{L}^3,$ with $M_t=X_t(M),$ evolves by the inverse mean curvature flow (IMCF) if it is a solution of
\begin{equation}\label{def-fcmi}
\left\{\begin{aligned}
 \frac{\partial}{\partial t}X(p,t)=-\frac{1}{H(p,t)}N(p,t)\\
X(p,0)=X(p),\hspace{0.5cm} p\in M,
\end{aligned}\right.
\end{equation}
where $H(\cdot,t)$ and $N(\cdot,t)$ are the mean curvature and the unitary normal vector field of $M_t,$ respectively, and $t\in[0,T)$. This flow was studied in \cite{luc12}, \cite{G1} and \cite{luc14}, where we can find examples and results about the existence and regularity
of extrinsic curvature flows on semi-Riemannian manifolds, with emphasis on Lorentzian flows.

The self-similar solutions, i.e., surfaces
that move by homothetic contractions or expansions, translations, or rotations under the flow, play a notable role in the theory of singularities of extrinsic curvature flows. In this article, we will focus on self-shrinking and self-expanding homothetic solitons. The case of ruled surfaces that are translating solitons of the IMCF was classified recently by the authors in \cite{SS}.

In Euclidean space, we can find recent studies on the homothetic solitons for the IMCF in \cite{luc1}, \cite{luc2}, \cite{luc3}, \cite{luc4} and \cite{luc5}. Although we could not find in the literature results about homothetic solitons of the inverse mean curvature flow in $\mathbb{L}^3,$ the understanding of these solitons is crucial to understanding the flow, both because they are the simplest solutions and because they usually appear as singularities of the flow. In the two-dimensional setting, however, we can cite the recent work of Tenenblat and da Silva \cite{keti}, where they consider the flow of curvature and inverse curvature (of curves) in the two-dimensional light cone, contained in three-dimensional Minkowski space, proving that ellipses and hyperbolas are the only curves that evolve by homotheties, among other specific properties for these solutions.

Although the theory of curvature flows was developed mostly for spacelike hypersurfaces, the special case of self-similar solutions makes sense for both spacelike and timelike hypersurfaces. In fact, we say that an immersion $X_0:M^2\longrightarrow\mathbb{L}^3$, with non-zero mean curvature $H$, is an homothetic self-similar solution of the IMCF if there exists a positive function $\phi(t)$ with $\phi(0)=1$ such that
\begin{equation}\label{condii}
X(\cdot,t)=\phi(t)X_0
\end{equation}
is a solution of \eqref{def-fcmi}. As we will prove later, the equation \eqref{condii} will be equivalent to $X_0$ satisfying the equation
\begin{equation}\label{selfsimilari}
    C\interno{N}{X_0}=-\frac{1}{H},    
\end{equation}
form some $C\in\mathbb{R}\backslash\{0\}.$ We will also show that the surface $X_0(M^2)$ contracts homothetically when $\epsilon C<0,$ and expands homothetically when $\epsilon C>0,$ where $\epsilon=\interno{N}{N}\in\{-1,1\}$. In the first case, we call $M^2$ a self-shrinker, and in the second case, we call $M^2$ a self-expander.

In this paper, we will classify homothetic self-similar solutions of the IMCF, assuming that $M^2$ is a ruled surface. A ruled surface is the union of an one parameter family of straight lines and can be parameterized locally by
\[
X(s,t)=\gamma(s)+t\beta(s),
\]
where $t\in\mathbb{R}$, $\gamma:I\subset\mathbb{R}\rightarrow\mathbb{L}^3$ is a smooth curve on $\mathbb{L}^3$ and $ \beta(s)$ is a field of vectors along $\gamma$. The curve $\gamma(s)$ is called the base curve of the surface $M^2$ and $\beta(s)$ are the generators of the surface. If $\gamma$ reduces to a point, the surface is called conic. On the other hand, if the vector field $\beta(s)$ is parallel to a fixed direction, for all $s$, i.e., $\beta(s)$ is a constant vector field, then the surface is called cylindrical.

Ruled surfaces in $\mathbb{L}^3$ were the object of study by several authors, such as \cite{Koba}, \cite{Vande}, \cite{Dillen}, \cite{dillen-k}, \cite{Kim}, \cite{Kim2}, \cite{Kim3}, \cite{Kim4}, \cite{liu}, and \cite{ali}, among others.

The article will be structured as follows: In the section \ref{prel}, we will introduce some basic concepts and results of differential geometry and linear algebra in $\mathbb{L}^3,$ and deduce the basic equations that will be useful in the proofs of the main theorems. In section \ref{noncyl} we classify non-cylindrical ruled surfaces that are homothetic solitons, and in section \ref{cyl}, we conclude the paper by classifying all ruled cylindrical surfaces that are homothetic solitons.
\section{Preliminaries}\label{prel}

In this section, for the sake of completeness, we will present the basic concepts about differential geometry of surfaces in $\mathbb{L}^3.$ More information can be found, for example, in \cite{book-CL}, \cite{lopez-survey} and \cite{oneill}.

Let $\mathbb{L}^3$ be the three-dimensional Lorentz-Minkowski space, i.e., the Euclidean space $\mathbb{R}^3$ with its vector space structure, endowed with the pseudo-Euclidean metric
\[
\interno{x}{y}=x_1y_1+x_2y_2-x_3y_3,
\]
where $x=(x_1,x_2,x_3)$ and $y=(y_1,y_2,y_3)$. The vectors $v\in\mathbb{L}^3$ can be of three causal types: spacelike, when $\interno{v}{v}>0$, or $v=0$; timelike, when $\interno{v}{v}<0$; and lightlike, when $\interno{v}{v}=0$, but $v\neq0$. Analogously, a nontrivial linear subspace $\mathcal{U}\subset\mathbb{L}^3$ can be classified as
\begin{enumerate}
\item[(i)] spacelike, if $\interno{.}{.}|_{\mathcal{U}}$ is positive definite;
\item[(ii)] timelike, if $\interno{.}{.}|_{\mathcal{U}}$ is negative-definite, or indefinite and nondegenerate;
\item[(iii)] lightlike, if $\interno{.}{.}|_{\mathcal{U}}$  is degenerate.
\end{enumerate}

As in Euclidean space, a set of vectors $\{v_i\}_{i\in I}$ is orthogonal if $\interno{v_i}{v_j}=0$ for every $i,j\in I,$ $i\neq j,$ and it is orthonormal if, additionally, $|v_i|=1$ for every $i\in I,$ where
\[
|v|=\sqrt{|\langle v,v\rangle|}.
\]
In the next lemma, for the sake of completeness, we list some basic properties about the orthogonality of vectors in $\mathbb{L}^3$ that will be useful in the proofs of our results:

\begin{lemma}\label{basic-1}
The vectors of $\mathbb{L}^3$ have the following basic properties:
\begin{itemize}
    \item[(i)] Two lightlike vectors of $\mathbb{L}^3$ are orthogonal if and only if they are linearly dependent;
    \item[(ii)] There are not two timelike orthogonal vectors in $\mathbb{L}^3;$
    \item[(iii)] There are no timelike vectors in $\mathbb{L}^3$ that are orthogonal to lightlike vectors.
\end{itemize}
\end{lemma}
It will also be useful to remember the concepts of vector product and mixed product of vectors in $\mathbb{L}^3.$ The vector product of $u$ and $v$ is the only vector $u\times v\in\mathbb{\mathbb{L}}^3$ such that
\[
\interno{u\times v}{x}=\det(x,u,v)\text{, for every }x\in\mathbb{L}^3.
\]
In its turn, the mixed product of the vectors $u,v,w\in\mathbb{L}^3$ is defined by
\begin{equation}\label{mixed}
(u,v,w):=\interno{u\times v}{w}=\det(u,v,w).
\end{equation}

As in Euclidean space, we also obtain a version of the Lagrange identity on $\mathbb{L}^3$, that is, if $u,v\in\mathbb{L}^3$ , then
\begin{equation}\label{Lagrange}
\interno{u\times v}{u\times v}=-\interno{u}{u}\interno{v}{v}+\interno{u}{v}^2.
\end{equation}

A smooth parametrized curve $\alpha:I\subset\mathbb{R}\rightarrow\mathbb{L}^3$ is regular if $\alpha'(t)\neq 0$ for every $t\in I.$ A regular curve $\alpha$ is called spacelike, timelike, or lightlike at $t_0\in I,$ if $\alpha'(t_0)$ is spacelike, timelike, or lightlike, respectively. Notice that the condition of being spacelike or timelike are open properties, i.e., if it holds for a point $t_0\in I,$ it still holds in an open interval around $t_0.$ Since our classification theorems will be essentially local (we will not use any global or topological properties), we can assume that the curves have a fixed causal type.

Given a surface $M\subset \mathbb{L}^3,$ parametrized locally by the map $X:U\subset\mathbb{R}^2\to\mathbb{L}^3,$ where $X=X(s,t),$ the first fundamental form of $M$ has the coefficients 
\[
E=\lan X_s,X_s\ran, \quad F=\lan X_s,X_t\ran \quad \mbox{and} \quad G=\lan X_t,X_t\ran.
\]
In this paper, we will deal with nondegenerate surfaces, i.e., those such that the quantity $|X_s\times X_t|=\sqrt{|EG-F^2|}\neq 0$ for all of its points.

A regular surface $M$ of $\mathbb{L}^3$ is spacelike, timelike, or lightlike, if its tangent plane $T_pM$ is spacelike, timelike, or lightlike, respectively, for every $p\in M.$ In $X(U)\subset M,$ it is defined two unit normal vector fields
\begin{equation}\label{Normal}
N=\pm\frac{X_s\times X_t}{|X_s\times X_t|}.    
\end{equation}
These normal vector fields satisfy 
\[
\lan N,N\ran=\epsilon,
\]
where $\epsilon=-1$ if $M$ is spacelike and $\epsilon=1$ if $M$ is timelike. The second fundamental form of $M$ in $\mathbb{L}^3$ is the bilinear and symmetric form $II_p:T_pM\times T_pM\to (T_pM)^\perp$ defined by
\begin{equation}
    \lan II_p(v,w),N(p)\ran=\lan -dN_p(v),w\ran,
\end{equation}
where $dN_p:T_pM\to T_pM$ is the differential of $N$ at $p.$ The second fundamental form can be expressed, in local coordinates, as
\[
\begin{aligned}
e&=\langle N,X_{ss}\rangle=\frac{(X_s,X_t,X_{ss})}{|X_s\times X_t|},\\
f&=\langle N,X_{st}\rangle=\frac{(X_s,X_t,X_{st})}{|X_s\times X_t|},\\
g&=\langle N,X_{ss}\rangle=\frac{(X_s,X_t,X_{tt})}{|X_s\times X_t|},\\
\end{aligned}
\]
where $(u,v,w)$ is the mixed product given by \eqref{mixed}.
If $M\subset\mathbb{L}^3$ is a nondegenerate surface and $N$ is its Gauss map, then its mean curvature is defined by
\[
H(p)=\frac{\epsilon}{2}\interno{\trace II_p}{N(p)}.
\]
In local coordinates, the mean curvature has the expression
 \begin{equation}\label{H-local}
 \begin{aligned}    
H&=\dfrac{\epsilon}{2}\dfrac{eG-2fF+Eg}{EG-F^2}\\
&=-\dfrac{1}{2}\dfrac{G(X_s,X_t,X_{ss})-2F(X_s,X_t,X_{st})+E(X_s,X_t,X_{tt})}{\vert EG-F^2\vert^{3/2}},
\end{aligned}
  \end{equation}
where we used that $\vert X_s\times X_t\vert^2=-\epsilon (EG-F^2)$.

An immersion $X_0:M^2\longrightarrow\mathbb{L}^3$, with nonzero mean curvature $H$, is an homothetic self-similar solution of the IMCF if there exists a positive function $\phi(t)$ with $\phi(0)=1$ such that
\begin{equation}\label{condi}
X(\cdot,t)=\phi(t)X_0
\end{equation}
is a solution of \eqref{def-fcmi}. Replacing $\eqref{condi}$ in \eqref{def-fcmi} we get
\begin{equation}\label{210}
    \phi'(t)X_0=-\frac{\phi(t)}{H(p,0)}N(p,0).
\end{equation}
Denoting by $H(p,0)=H(p),$ $N(p,0)=N(p),$ taking the inner product of \eqref{210} with $N,$ and remembering that $\epsilon=\interno{N}{N}\in\{-1,1\}$ we get
\begin{equation}\label{homotetic}
    \frac{\phi'(t)}{\phi(t)}\interno{N(p)}{X_0(p)}=-\frac{\epsilon}{H(p)}.
\end{equation}
Therefore, $\phi'(t)/\phi(t)$ is constant, i.e., there exists $C\in\mathbb{R}$ such that
\begin{equation}\label{edo-phi}
\frac{\phi'(t)}{\phi(t)}=\epsilon C,
\end{equation}
whose solution is
\begin{equation}
    \phi(t)=e^{\epsilon Ct}.
\end{equation}
This implies that, for $\epsilon C<0,$ $X_0(M^2)$ shrinks homothetically under the IMCF and, for $\epsilon C>0,$ $X_0(M^2)$ expands homothetically under the IMCF. In the first case, we call $M^2$ a self-shrinker, and in the second case, we call $M^2$ a self-expander.
Furthermore, by replacing \eqref{edo-phi} into \eqref{homotetic}, we obtain
\begin{equation}\label{selfsimilar}
    C\interno{N}{X}=-\frac{1}{H},
\end{equation}
which will be called the equation of the homothetic self-similar solutions of the IMCF.

Now, consider a ruled surface in $\mathbb{L}^3$ parametrized by $X(s,t)=\gamma(s)+t\beta(s)$. With this parametrization, we have
\[
X_t=\beta,\ X_s=\gamma'+t\beta',\ X_{ss}=\gamma''+t\beta'',\ X_{st}=\beta'\ \mbox{and}\ X_{tt}=0.
\] 
Since $X_{tt}=0,$ the mean curvature equation \eqref{H-local} can be rewritten as
\begin{equation}\label{H-local-2}
H=-\frac{1}{2}\displaystyle\frac{G(X_s,X_t,X_{ss})-2F(X_s,X_t,X_{st})}{\vert EG-F^2\vert^{3/2}}.    
\end{equation}
Replacing \eqref{H-local-2} in \eqref{selfsimilar}, we obtain
$$\frac{C}{\sqrt{\vert EG-F^2\vert}}(X_s,X_t,X)\left[\frac{G(X_s,X_t,X_{ss})-2F(X_s,X_t,X_{st})}{\vert EG-F^2\vert^{3/2}}\right]=2$$
or, equivalently, at the nondegenerate points,
\begin{equation}\label{princ}
C(X_s,X_t,X)[G(X_s,X_t,X_{ss})-2F(X_s,X_t,X_{st})]=2(EG-F^2)^2.
\end{equation} 
This will be the equation we analyze in this paper.

\section{Non-cylindrical homothetic Self-Similar Solutions}\label{noncyl}

In this section, we classify all the homothetic solutions of the IMCF that are nondegenerate, non-cylindrical, ruled surfaces, i.e., the solutions of \eqref{princ} with parametrization $X(s,t)=\gamma(s)+t\beta(s)$ where $\gamma:I\subset\mathbb{R}\longrightarrow\mathbb{L}^3$ is a regular curve and $\beta'(s)\neq0$ in the interval $I$. We start with the case when $\beta$ is (locally) a lightlike vector field, i.e., $\interno{\beta}{\beta}=0.$ Since $X_t(s,t)=\beta(s)$ is a lightlike vector field of the tangent space of $M,$ we have that $M$ is always timelike.

\begin{theorem}\label{thm-light}
Let $M$ be a non-cylindrical, nondegenerate, ruled surface of $\mathbb{L}^3,$ parametrized by $X(s,t)=\gamma(s)+t\beta(s)$, such that $\beta$ is a lightlike vector field parametrized by the arc length, i.e., such that $\interno{\beta'(s)}{\beta'(s)}=1$. If $M$ is a homothetic self-similar solution for the inverse mean curvature flow, then
\begin{itemize}
    \item[(i)] $M$ is a timelike self-expander with $C=1;$
    \item[(ii)] $\gamma(s)=a(s)\beta(s)+b(s)\beta'(s)$ where $a(s)$ and $b(s)\neq 0$ are real smooth functions;
    \item[(iii)] $(\beta''(s),\beta(s),\beta'(s))=1.$ 
\end{itemize}
Conversely, let $M$ be a ruled surface parametrized by $X(s,t)=\gamma(s)+t\beta(s),$ where $\beta(s)$ is a lightlike vector field, parametrized by the arc length. If $\beta$ and $\gamma$ satisfy (ii) and (iii), for any real functions $a(s)$ and $b(s)\neq 0,$ then $M$ is a timelike self-expander of the inverse mean curvature flow with $C=1$.

\end{theorem}
\begin{proof}
If $\beta$ is a lightlike vector field, i.e., $\interno{\beta}{\beta}=0$ then $\interno{\beta'}{\beta}=0$. In this way, we can choose a parametrization such that $\interno{\beta'}{\beta'}=1.$ It follows from the Lagrange identity \eqref{Lagrange}, that
$$\interno{\beta'\times \beta}{\beta'\times \beta}=-\interno{\beta'}{\beta'}\interno{\beta}{\beta}+\interno{\beta'}{\beta}^2=0,$$
i.e., $\beta'\times\beta$ is a lightlike vector field. This gives, by using Lemma \ref{basic-1}, that there exists a function $k(s)$ such that
\begin{equation}\label{vec-prod-beta}
    \beta'(s)\times \beta(s)=k(s)\beta(s).
\end{equation}
Taking derivatives in \eqref{vec-prod-beta}, we obtain
\[
\beta''\times \beta=k'\beta+k\beta',
\]
that gives
\[
\interno{\beta''\times \beta}{\beta''\times \beta}=k^2.
\]
On the other hand, by using the Lagrange identity \eqref{Lagrange},
\[
\interno{\beta''\times \beta}{\beta''\times \beta}=-\interno{\beta''}{\beta''}\interno{\beta}{\beta}+\interno{\beta''}{\beta}^2=1,
\]
i.e.,
\[
k(s)^2=1.
\]
By taking $\tilde{\beta}(s)=-\beta(s)$ if necessary (notice that, in this case, we obtain the same ruled surface), we can assume that $k(s)=1,$ i.e.,
\begin{equation}\label{vec-prod-beta-2}
    \beta'(s)\times \beta(s)=\beta(s).
\end{equation}
The coefficients of the first fundamental form of $M$ are given by
 \[
\begin{cases}
E&=\interno{\gamma'}{\gamma'}+2t\interno{\gamma'}{\beta'}+t^2,\\
F&=\interno{\gamma'}{\beta},\\
G&=\interno{\beta}{\beta}=0\\
\end{cases}
 \]
 which implies that $F=\interno{\gamma'}{\beta}\neq 0$, since the surface is nondegenerate. It follows, from \eqref{princ}, that
 $$C(X_s,X_t,X)[-2F(X_s,X_t,X_{st})]=2(-F^2)^2$$
 i.e.,
 \begin{equation}\label{II}
	C(X_s,X_t,X)(X_s,X_t,X_{st})=-F^3.
	\end{equation}
 Since
 $$X_s\times X_t=\gamma'\times\beta+t\beta'\times\beta,
 $$
 we obtain
\begin{equation}\label{eq.luz-1}      (X_s,X_t,X)=\interno{\gamma'\times\beta+t\beta'\times\beta}{\gamma+t\beta}=(\gamma',\beta,\gamma)+t(\beta',\beta,\gamma)
 \end{equation}
and
\begin{equation}\label{eq.luz-2} 
     (X_s,X_t,X_{st})=(\gamma',\beta,\beta')+t(\beta',\beta,\beta')=(\gamma',\beta,\beta').
 \end{equation}
 Replacing \eqref{eq.luz-1} and \eqref{eq.luz-2} in \eqref{II}, we have
 \begin{equation}\label{32}  C[(\gamma',\beta,\gamma)+t(\beta',\beta,\gamma)](\gamma',\beta,\beta')=-\interno{\gamma'}{\beta}^3,
 \end{equation} 
 which is equivalent to an identically zero polynomial, in the variable $t,$ with coefficients depending on $s,$ given by
 \begin{equation}\label{pol}
    p(t):=C(\gamma',\beta,\gamma)(\gamma',\beta,\beta')+\interno{\gamma'}{\beta}^3+tC(\beta',\beta,\gamma)(\gamma',\beta,\beta')=0.
\end{equation}
It follows that
\begin{equation}
\begin{cases}\label{Pol-partes}
   C(\beta',\beta,\gamma)(\gamma',\beta,\beta')&=0\\
C(\gamma',\beta,\gamma)(\gamma',\beta,\beta')+\interno{\gamma'}{\beta}^3&=0. 
\end{cases}    
\end{equation}
The first equation of \eqref{Pol-partes} gives us two possibilities: $(\gamma',\beta,\beta')=0,$ which implies, from the second equation of \eqref{Pol-partes}, that $F=\interno{\gamma'}{\beta}=0,$ which leads us to a contradiction, or $(\beta',\beta,\gamma)=0$ which gives that $\{\beta',\beta,\gamma\}$ is linearly dependent. In this last case, there exist smooth functions $a(s)$ and $b(s)$ such that 
\begin{equation}\label{ld}
    \gamma(s)=a(s)\beta(s)+b(s)\beta'(s).
\end{equation}
By taking the derivative of \eqref{ld}, we have
 \begin{equation}\label{dld}
     \gamma'=a'\beta+(a+b')\beta'+b\beta''.
 \end{equation}
This gives
 \begin{equation}
 \label{vet}
  \gamma'\times\beta=(a+b')(\beta'\times\beta)+b(\beta''\times \beta) =(a+b')\beta+b(\beta''\times \beta). 
 \end{equation}
 Taking the inner product of \eqref{vet} with $\gamma$ we obtain
\begin{equation}\label{mgamma}      (\gamma',\beta,\gamma)=b^2(\beta'',\beta,\beta').
 \end{equation}
 Similarly, taking the inner product of \eqref{vet} with $\beta',$ we have
 \begin{equation}\label{mbeta}
(\gamma',\beta,\beta')=b(\beta'',\beta,\beta').
  \end{equation}
Notice that  $\interno{\gamma'}{\beta}=b\interno{\beta''}{\beta}$
and, by taking the derivative on $\interno{\beta}{\beta'}=0,$ it holds $\interno{\beta}{\beta''}=-1.$ Thus, it follows that
\begin{equation}\label{gbprinc}
   \interno{\gamma'}{\beta}=-b\neq0.  
 \end{equation}
Replacing the expressions obtained in $\eqref{mgamma}$, $\eqref{mbeta}$ and \eqref{gbprinc} in the second equation of the system \eqref{Pol-partes} we have
 \begin{equation}
b^3C(\beta'',\beta,\beta')^2-b^3=0,
 \end{equation}
 or, equivalently,
 
\begin{equation}\label{expander}   C(\beta'',\beta,\beta')^2=1. 
\end{equation}
Since, by \eqref{vec-prod-beta-2}, 
\[
(\beta'',\beta,\beta')=-(\beta',\beta,\beta'')=-\interno{\beta}{\beta''}=1,
\]
we obtain that $C=1.$ Thus, since $\epsilon=1$, because $M$ is timelike, we conclude that $M$ is a self-expander.

Conversely, if $M$ is a ruled surface parametrized by $X(s,t)=\gamma(s)+t\beta(s),$ and such that $\beta(s)$ is lightlike, $\interno{\beta'(s)}{\beta'(s)}=1,$ $(\beta''(s),\beta(s),\beta'(s))=1,$ and $\gamma(s)=a(s)\beta(s)+b(s)\beta'(s),$ for arbitrary real valued functions $a(s)$ and $b(s)\neq 0,$  then, by replacing these informations in \eqref{II} and by following the steps of the previous proof, we can conclude that $M$ is a timelike self-expander with $C=1.$

\end{proof}

\begin{example}
The class of lightlike vector fields $\beta(s),$ parametrized by arc length, such that 
\[
(\beta''(s),\beta(s),\beta'(s))=1
\]
is not empty. In fact, take the vectors
\begin{equation}\label{ABC}
\vec{A}=(0,a_0,a_0), \ \vec{B}=(b_0,b_0,b_0), \ \mbox{and}\ \vec{C}=(c_0,0,c_0),
\end{equation}
where $a_0,b_0,c_0>0$ are real numbers to be determined. Define
\begin{equation}\label{beta-quad}
\beta(s)=\vec{A}s^2+\vec{B}s+\vec{C}. 
\end{equation}
We have
\[
\aligned
\interno{\beta(s)}{\beta(s)}&=\interno{\vec{A}}{\vec{A}}s^4 + 2\interno{\vec{A}}{\vec{B}}s^3 + \left(2\interno{\vec{A}}{\vec{C}}+\interno{\vec{B}}{\vec{B}}\right)s^2\\
&\quad+ 2\interno{\vec{B}}{\vec{C}}s + \interno{\vec{C}}{\vec{C}}\\
&\quad =(-2a_0c_0+b_0^2)s^2=0
\endaligned
\]
if and only if $b_0^2=2a_0c_0.$ On the other hand, since $\beta'(s)=2\vec{A}s+\vec{B},$ we must have
\[
\interno{\beta'(s)}{\beta'(s)}=\interno{\vec{B}}{\vec{B}}=b_0^2=1,
\]
i.e., we have the relation
\[
2a_0c_0=b_0^2=1.
\]
In its turn, we have
\[
\aligned
(\beta''(s),\beta(s),\beta'(s))&=(2\vec{A},\vec{A}s^2+\vec{B}s+\vec{C},2\vec{A}s+\vec{B})=2(\vec{A},\vec{C},\vec{B})\\
&=2\left|
\begin{array}{ccc}
   0& a_0 & a_0 \\
   c_0& 0 & c_0 \\
   b_0&b_0&b_0\\
\end{array}
\right|\\
&=2a_0b_0c_0=b_0^3=1,\\
\endaligned
\]
if and only if $b_0=1.$ 

Therefore, defining $\beta$ as in \eqref{beta-quad}, where $\vec{A},$ $\vec{B},$ and $\vec{C}$ are given by \eqref{ABC}, with $2a_0c_0=1$ and $b_0=1$ gives an example of a lightlike vector field that satisfies the condition of the converse of Theorem \ref{thm-light}. We can get another example by taking
$\vec{A}=(a_0,0,a_0),$ $\vec{B}=(-1,-1,-1)$ and $\vec{C}=(0,c_0,c_0),$ where $2a_0c_0=1.$

This will give examples of timelike ruled self-expanders of the IMCF with $C=1,$ by taking the parametrization
\[
X(s,t)=(a(s)+t)\beta(s) + b(s)\beta'(s),
\]
for $a(s)$ and $b(s)\neq 0$ arbitrary functions (see Figure \ref{exp-lightlike-1}).

\begin{figure}[ht]
	\begin{center}
		\def\svgwidth{0.7\textwidth} 
		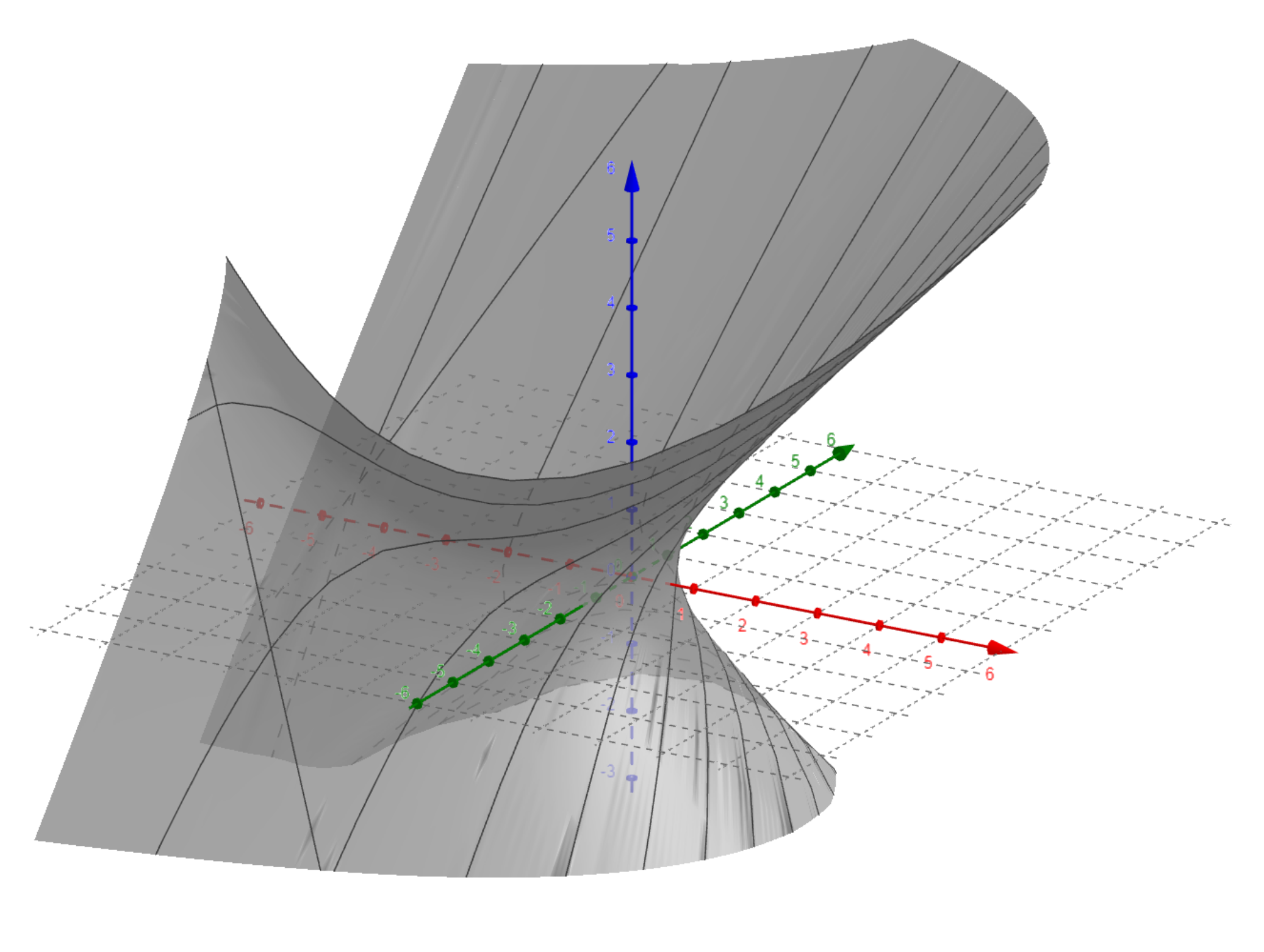
		\caption{A non-cylindrical timelike self-expander of the IMCF in $\mathbb{L}^3,$ for $C=1,$ in the conditions of Theorem \ref{thm-light}. Here we are taking $a(s)=s,$ $b(s)=s^2+1,$ $\beta(s)=(0,1,1)s^2+(1,1,1)s+\left(\frac12,0,\frac12\right).$} \label{exp-lightlike-1}
	\end{center}
\end{figure}
\end{example}

\begin{example}
    Taking $\beta(s)=(-\cos s,-\sin s,-1)$ we obtain another class of ruled surfaces satisfying the hypothesis of Theorem \ref{thm-light} (see Figure \ref{exp-lightlike-2}). 
\begin{figure}[ht]
	\begin{center}
		\def\svgwidth{0.7\textwidth} 
		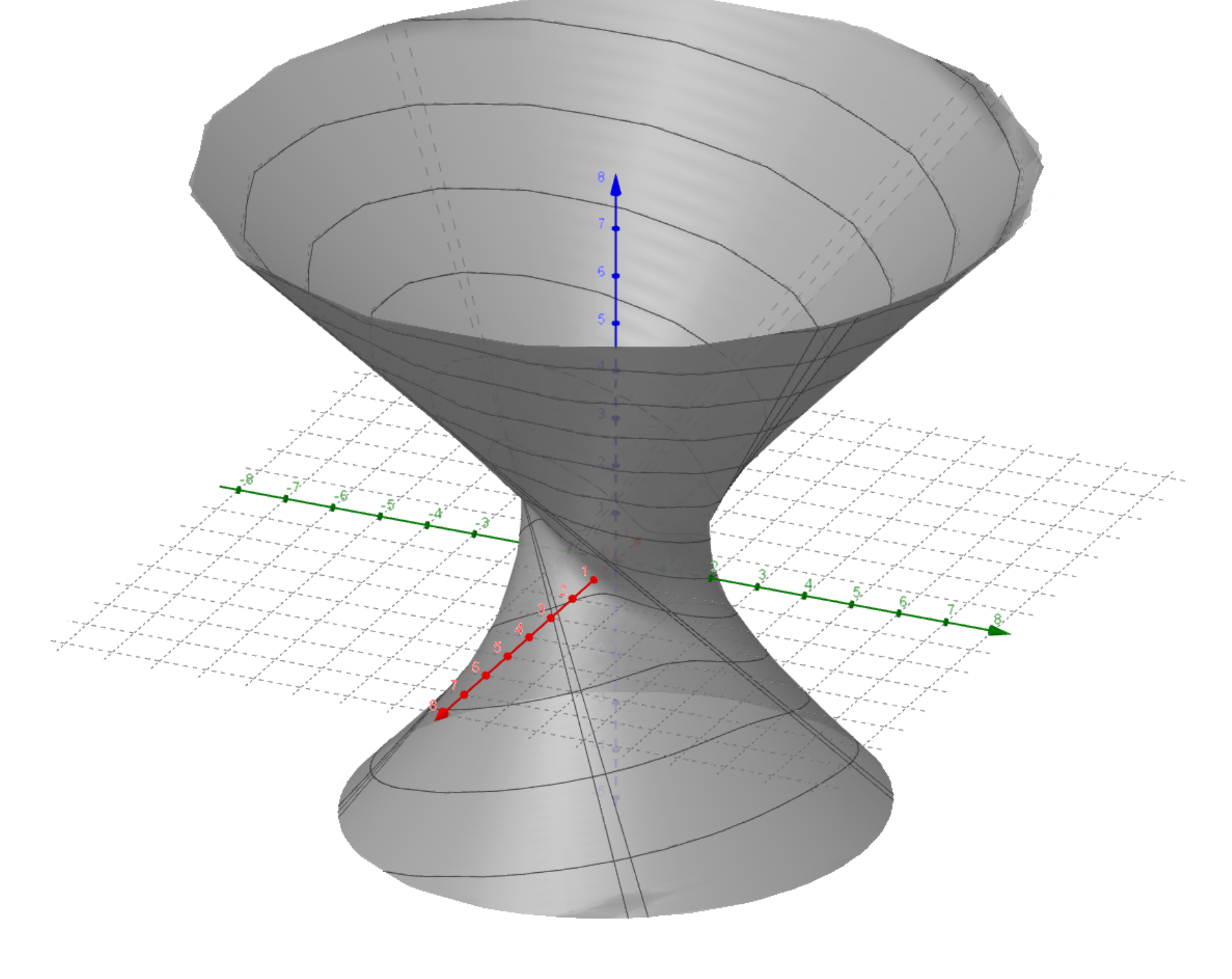
		\caption{Another non-cylindrical timelike self-expander of the IMCF in $\mathbb{L}^3,$ for $C=1,$ in the conditions of Theorem \ref{thm-light}. Here we are taking $\beta(s)=(-\cos s, -\sin s,-1),$ $a(s)=\cos s,$ and $b(s)=1+\sin^2 s.$}\label{exp-lightlike-2}
	\end{center}
\end{figure}
\end{example}

\begin{remark}
    It was proved by Dillen and K\"uhnel in \cite{Dillen} (see Theorem 2, item (iii), p.313) that ruled surfaces of $\mathbb{L}^3$ with lightlike $\beta(s)$ are Weingarten surfaces satisfying $H^2=K,$ where $K$ is the Gaussian curvature of the surface.
\end{remark}

We will continue the classification considering now the case in which $\beta(s)$ is not a lightlike vector field, i.e., such that $\beta(s)\neq 0$ in an open interval.

\begin{theorem}\label{thm-non-cyl}
    Let $M$ be a non-cylindrical smooth ruled surface $\mathbb{L}^3,$ parametrized by $X(s,t)=\gamma(s)+t\beta(s)$, such that $\beta$ is not a lightlike vector field. If $M$ is a self-similar solution for the inverse mean curvature flow, then there exists a parameter $s$ such that 
    \begin{itemize}
    \item[(i)] $\beta$ is a straight line parametrized by $\beta(s)=(1,s,s);$
     \item[(ii)] If $\gamma(s)=(x(s),y(s),z(s)),$ then
     \begin{equation}
    \left\{\aligned
   x(s)&=-\dfrac{C}{2C-8}\left(\dfrac{(C-8)k_1}{C}\right)^{\frac{C}{C-8}}s^{\left(\frac{2C-8}{C-8}\right)}+k_2\\
    y(s)&=\left(\frac{s^2+1}{2}\right)\left(\frac{(C-8)k_1s}{C}\right)^{\frac{C}{C-8}}\\
&\quad-\dfrac{C}{2C-8}\left(\dfrac{(C-8)k_1}{C}\right)^{\frac{C}{C-8}} s^{\left(\frac{3C-16}{C-8}\right)}+k_2s\\
    z(s)&=\left(\frac{s^2-1}{2}\right)\left(\frac{(C-8)k_1s}{C}\right)^{\frac{C}{C-8}}\\
&\quad-\dfrac{C}{2C-8}\left(\dfrac{(C-8)k_1}{C}\right)^{\frac{C}{C-8}} s^{\left(\frac{3C-16}{C-8}\right)}+k_2s,\\
\endaligned\right.
\end{equation}
when $C\neq 8$ (see Figure \ref{fig-non-cyl-1}) and
\begin{equation}
    \left\{\aligned
   x(s)&=-e^{k_1s}\left(s-\frac{1}{k_1}\right)+k_2\\
    y(s)&=e^{k_1s}\left(\frac{1-s^2}{2}+\frac{s}{k_1}\right)+k_2s\\
    z(s)&=e^{k_1s}\left(-\frac{1+s^2}{2}+\frac{s}{k_1}\right)+k_2s\\
\endaligned\right.
\end{equation}
when $C=8$ (see Figure \ref{fig-non-cyl-2}), where $k_1,k_2$ are constants with $k_1\neq0$.
    \end{itemize}
    \end{theorem}
  
\begin{figure}[htb!]
	\begin{center}
		\def\svgwidth{0.9\textwidth} 
		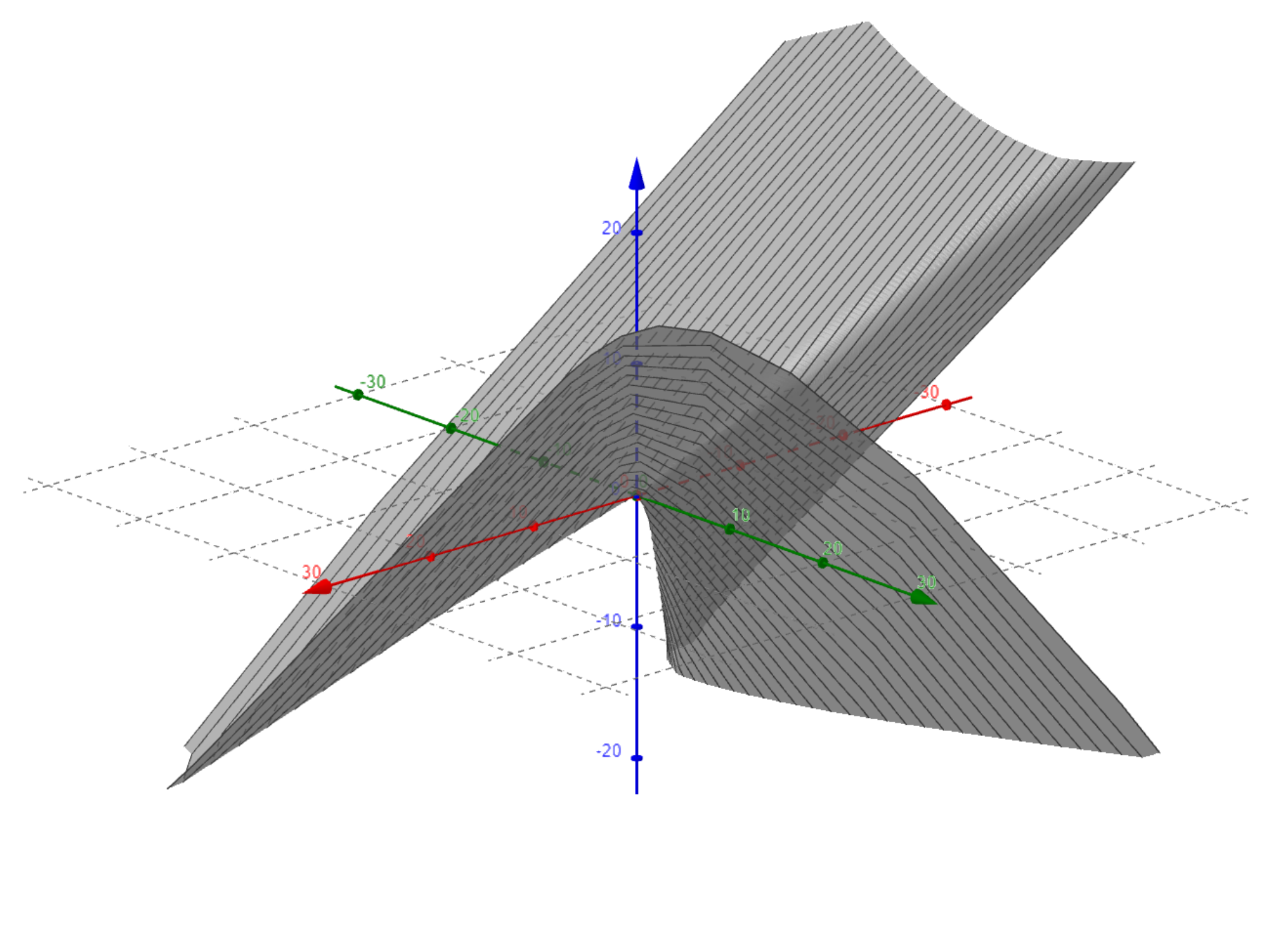
		\caption{Non-cylindrical self-similar solutions of the IMCF given by Theorem \ref{thm-non-cyl}. Here we are taking $C=k_1=9$ and $k_2=0$.} 
	\end{center}\label{fig-non-cyl-1}
\end{figure}
   
   \begin{proof}
       Suppose that the parametrization $X(s,t)=\gamma(s)+t\beta(s)$ is orthogonal, i.e., $\interno{\beta}{\gamma'}=0.$ Since $\beta$ is not lightlike, by hypothesis, we can choose the parameter $s$ such that $\interno{\beta}{\beta}=:\delta\in\{-1,1\}.$ This gives $\interno{\beta}{\beta'}=0.$ The coefficients of the first fundamental form are
\[
 \begin{cases} 
  E&=\interno{\gamma'}{\gamma'}+2\interno{\gamma'}{\beta'}t+\interno{\beta'}{\beta'}t^2,\\ F&=\interno{\beta}{\gamma'}=0,\\
  G&=\delta.\\
\end{cases}
\]
On the other hand, 
$$
(X_s,X_t,X_{ss})=(\gamma',\beta,\gamma'')+[(\gamma',\beta,\beta'')+(\beta',\beta,\gamma'')]t+(\beta',\beta,\beta'')t^2.
$$	
Replacing these facts in \eqref{princ}, we obtain
\begin{equation} \label{cond}	
   \begin{aligned}
   &C[(\gamma',\beta,\gamma)+t(\beta',\beta,\gamma)]\times\\
   &\times [\delta\{(\gamma',\beta,\gamma'')+[(\gamma',\beta,\beta'')+(\beta',\beta,\gamma'')]t+(\beta',\beta,\beta'')t^2\}]\\
   &=2[\delta(\interno{\gamma'}{\gamma'}+2\interno{\gamma'}{\beta'}t+\interno{\beta'}{\beta'}t^2)]^2.
\end{aligned}
   \end{equation}
   Making the calculations with the Equation \eqref{cond}, we obtain the identically zero fourth degree polynomial $p(t)=\sum_{i=0}^{4}A_i(s)t^i\equiv0,$ where
   \begin{equation}\label{c-pol}
\left\{\begin{aligned}
    A_0&=2\interno{\gamma'}{\gamma'}^2-C\delta(\gamma',\beta,\gamma)(\gamma',\beta,\gamma''),\\
    A_1&=8\interno{\gamma'}{\gamma'}\interno{\gamma'}{\beta'}\\
    &\quad-C\delta[(\gamma',\beta,\gamma)(\beta',\beta,\gamma'')+(\gamma',\beta,\gamma)(\gamma',\beta,\beta'')+(\beta',\beta,\gamma)(\gamma',\beta,\gamma'')],\\
    A_2&=4\interno{\beta'}{\beta'}\interno{\gamma'}{\gamma'}+8\interno{\gamma'}{\beta'}^2\\
    &\quad-C\delta[(\gamma',\beta,\gamma)(\beta',\beta,\beta'')+(\beta',\beta,\gamma)(\beta',\beta,\gamma'')+(\beta',\beta,\gamma)(\gamma',\beta,\beta'')],\\
    A_3&=8\interno{\gamma'}{\beta'}\interno{\beta'}{\beta'}-C\delta(\beta',\beta,\gamma)(\beta',\beta,\beta''),\\
    A_4&=2\interno{\beta'}{\beta'}^2.\\
\end{aligned}\right.
\end{equation}
By using that $p(t)$ is an identically zero polynomial, we obtain that $A_4=0,$ i.e., $\interno{\beta'}{\beta'}=0.$ Since $\interno{\beta}{\beta}=\delta,$ $\interno{\beta}{\beta'}=0,$ and there is no lightlike vector orthogonal to timelike vectors in $\mathbb{L}^3$ (see Lemma \ref{basic-1}, item (iii)), we have that $\beta$ is spacelike and then $\delta=1.$ Thus, $\beta'$ is a lightlike direction in the hyperboloid
$$
\{p\in\mathbb{L}^3:\interno{p}{p}=1\}.
$$
This implies that $\beta$ is a straight line, i.e.,  there exist vectors $\vec{a},\vec{b}\in\mathbb{L}^3$ such that $\beta(s)=\vec{a}s+\vec{b}$. Since $\beta'=\vec{a}$, we have 
\[
\interno{\vec{a}}{\vec{a}}=0=\langle\vec{a},\vec{b}\rangle\ \mbox{e}\ \langle\vec{b},\vec{b}\rangle=1.
\]
Therefore, up to rigid motions in $\mathbb{L}^3,$ we can consider $\vec{a}=(0,1,1)$ and $\vec{b}=(1,0,0)$. This implies that $\beta(s)=(1,s,s)$ and it proves item (i). 
Using item (i), we will consider the local parametrization of $M$ given by 				
\[
X(s,t)=\gamma(s)+t(1,s,s).
\] 
Using that $\beta(s)=(1,s,s)$ in \eqref{c-pol}, we obtain that $A_4=A_3=0$ and the system of equations
\begin{equation}\label{eq:A}
\left\{\begin{aligned}
  A_0&=2\interno{\gamma'}{\gamma'}^2-C(\gamma',\beta,\gamma)(\gamma',\beta,\gamma''),\\
  A_1&=8\interno{\gamma'}{\gamma'}\interno{\vec{a}}{\gamma'}-C[(\gamma',\beta,\gamma)(\vec{a},\beta,\gamma'')+(\vec{a},\beta,\gamma)(\gamma',\beta,\gamma'')],\\
  A_2&=8\interno{\vec{a}}{\gamma'}^2-C(\vec{a},\beta,\gamma)(\vec{a},\beta,\gamma'').
\end{aligned}\right.
\end{equation}
We claim that $\interno{\vec{a}}{\gamma'}\neq 0$. In fact, otherwise, considering \[\gamma(s)=(x(s),y(s),z(s)),\] if $y'-z'=\interno{\vec{a}}{\gamma'}=0$, it would follow, from the condition $\interno{\gamma'}{\beta}=0,$ that  
$$
\interno{\gamma'}{\beta}=\langle\vec{b},\gamma'\rangle+s\interno{\vec{a}}{\gamma'}=\langle\vec{b},\gamma'\rangle=x'=0,
$$
and this would imply that 
$$
\interno{\gamma'}{\gamma'}=(x')^2+(y')^2-(z')^2=0,
$$
i,e, $\gamma$ is lightlike, giving $EG-F^2=0,$ which is an absurd since $M$ is nondegenerate. Thus, $\interno{\vec{a}}{\gamma'}\neq 0$.

On the other hand, noting that
\begin{equation}
    \vec{a}\times \beta=\left|\begin{array}{rcr}
				e_1 & e_2  & -e_3 \\ 
				0 & 1 & 1\\
				1 & s  & s
				\end{array} \right|=(0,1,1)=\vec{a},
\end{equation}
and replacing this fact in the expression of $A_2=0$ given in \eqref{eq:A}, we have
$$C\interno{\vec{a}}{\gamma}\interno{\vec{a}}{\gamma''}-8\interno{\vec{a}}{\gamma'}^2=0,$$
i.e.,
 \begin{equation}\label{A2}
       C(y(s)-z(s))(y''(s)-z''(s))-8(y'(s)-z'(s))^2=0.
   \end{equation}
   Replacing $u(s):=y(s)-z(s)$ in \eqref{A2}, we get the ODE
   \begin{equation}\label{edoS}
     Cu(s)u''(s)-8u'(s)^2=0.
 \end{equation}
Since the solution $u(s)$ constant implies $\interno{\vec{a}}{\gamma'}=0,$ which is not possible, we will consider only the non-constant solutions of \eqref{edoS}. Defining $v(u)=u'(s),$ we obtain that \eqref{edoS} is equivalent to
\begin{equation}
    Cuv'(u)=v(u),
\end{equation}
whose solution is
\[
v(u)=k_1u^{\frac{8}{C}},
\]
which gives 
\[
u'(s)=k_1u(s)^{\frac{8}{C}},
\]
where $k_1\neq 0.$ This implies, after a translation of the parameter $s$, that
 \begin{equation}\label{exp1:y-z}
  y(s)-z(s)= u(s)=
  \begin{cases}
      \left(\dfrac{(C-8)k_1s}{C}\right)^{\frac{C}{C-8}}&,\ \mbox{for}\ C\neq 8,\\
      e^{k_1s}&,\ \mbox{for}\ C= 8.\\
  \end{cases}
\end{equation}
Since $\interno{\gamma'}{\beta}=0,$ we have
\begin{equation}\label{xp}
    	x'(s)=-(y'(s)-z'(s))s,
\end{equation}
which implies, by \eqref{exp1:y-z}, that
\begin{equation}\label{exp-x}
x(s)=
\begin{cases}
-\dfrac{C}{2C-8}\left(\dfrac{(C-8)k_1}{C}\right)^{\frac{C}{C-8}}s^{\left(\frac{2C-8}{C-8}\right)}+k_2&,\ \mbox{for} \ C\neq 8,\\
-e^{k_1s}\left(s-\dfrac{1}{k_1}\right)+k_2&,\ \mbox{for} \ C=8,
\end{cases}
\end{equation}
where $k_2\in\mathbb{R}$ is a constant.

The expressions $A_ 0=A_ 2=0$ given in $\eqref{eq:A}$, imply
\begin{equation}\label{A0A2}
   C(\gamma',\beta,\gamma)(\gamma',\beta,\gamma'')=2\langle \gamma',\gamma'\rangle^2 \quad \mbox{and} \quad C(\vec{a},\beta,\gamma)(\vec{a},\beta,\gamma'')=8\langle\vec{a},\gamma'\rangle^2.  
\end{equation}

On the other hand, multiplying the expression of $A_1=0,$ given in \eqref{eq:A}, by $\frac{1}{2}(\vec{a},\beta,\gamma)\cdot(\gamma',\beta,\gamma)$  and replacing the resulting expressions in \eqref{A0A2}, we have
\[\aligned
4(\gamma',\beta,\gamma)^2\interno{\vec{a}}{\gamma'}^2&+(\vec{a},\beta,\gamma)^2
\interno{\gamma'}{\gamma'}^2\\
&-4\interno{\gamma'}{\gamma'}\interno{\vec{a}}{\gamma'}(\gamma',\beta,\gamma)(\vec{a},\beta,\gamma)=0,
\endaligned
\]
i.e.,
\[    (2(\gamma',\beta,\gamma)\interno{\vec{a}}{\gamma'}-(\vec{a},\beta,\gamma)\interno{\gamma'}{\gamma'})^2=0,
\]
which gives
\begin{equation}\label{qperfeito}
    2(\gamma',\beta,\gamma)\interno{\vec{a}}{\gamma'}=(\vec{a},\beta,\gamma)\interno{\gamma'}{\gamma'}.
\end{equation}
Using again the notation $u(s):=y(s)-z(s)$ and \eqref{xp}, we have
\[
\interno{\vec{a}}{\gamma'}=u'(s),\quad
(\vec{a},\beta,\gamma)=\interno{\vec{a}}{\gamma}=u(s),
\]
\[
\aligned
(\gamma',\beta,\gamma)&=\left|
\begin{array}{ccc}
    x'&y'&z'\\
    1&s&s\\
   x&y&z\\
\end{array}
\right|\\
&=x's(z-y)-y'(z-sx)+z'(y-sx)\\
&=-sx'u+sx(y'-z')+z'y-y'z\\
&=s^2uu'+sxu'+(y-u)'y-y'(y-u)\\
&=su'(su+x)-yu'+y'u,
\endaligned
\]
and
\[
\aligned
\interno{\gamma'}{\gamma'}&=(x')^2+(y')^2-(z')^2\\
&=(x')^2+(y')^2-(y'-u')^2\\
&=(x')^2-(u')^2+2u'y'\\
&=(u')^2(s^2-1)+2u'y'.
\endaligned
\]
Replacing the expressions in \eqref{qperfeito} we get
\[
\aligned
2u'[su'(x+su)+uy'-yu']=u[(u')^2(s^2-1)+2u'y'].
\endaligned
\]
This gives, after simplifications,
\begin{equation}\label{y-exp}
    y(s)=\left(\frac{s^2+1}{2}\right)u(s)+sx(s),
\end{equation}
and, since $z(s)=y(s)-u(s),$ that
\begin{equation}\label{z-exp}
    z(s)=\left(\frac{s^2-1}{2}\right)u(s)+sx(s).
\end{equation}
The result follows by replacing the expressions for $x(s)$ and $u(s)$ given in \eqref{exp-x} and \eqref{exp1:y-z}, depending on $C\neq 8$ or $C=8.$
   \end{proof} 

\begin{figure}[ht]
	\begin{center}
\def\svgwidth{0.9\textwidth} 
		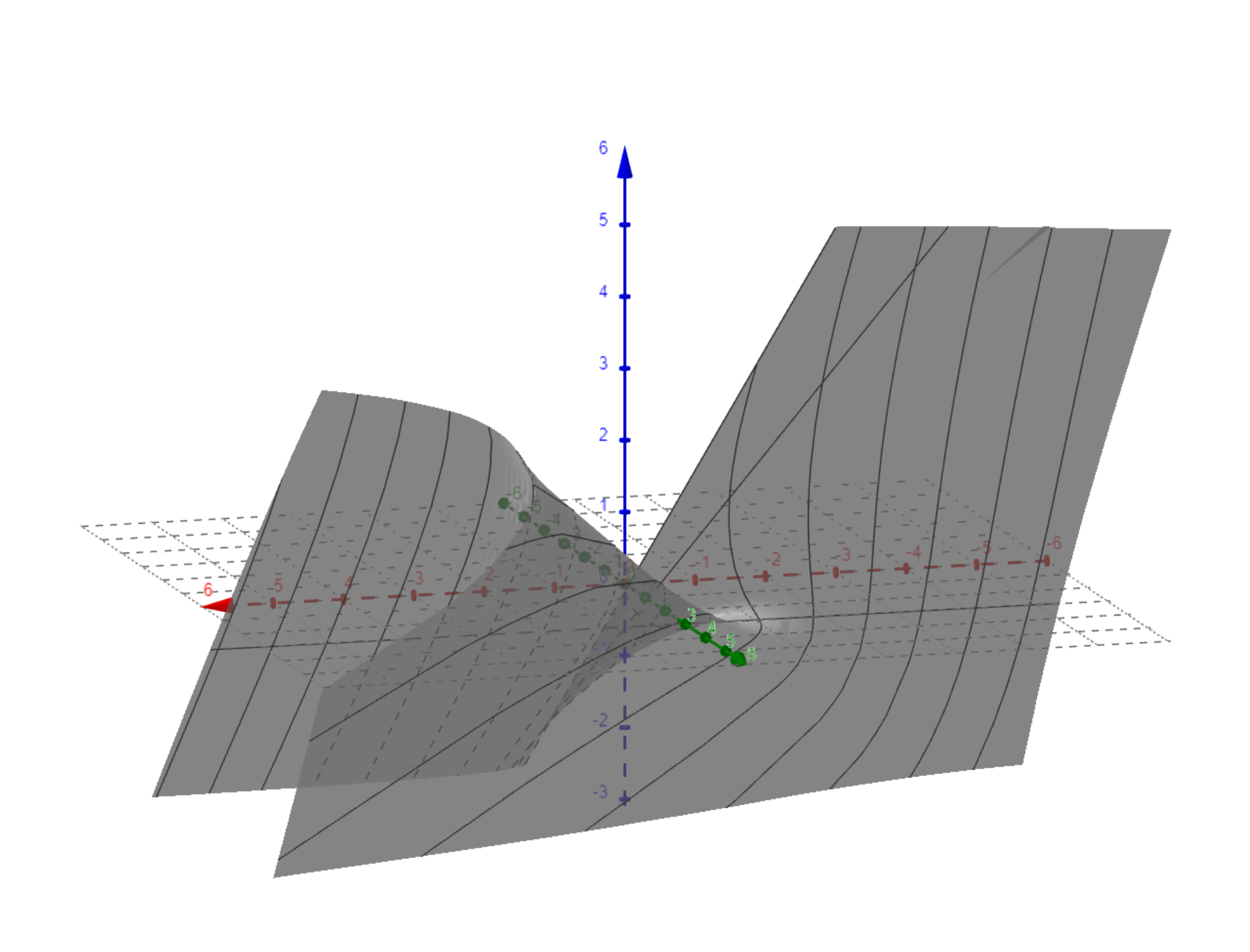
		\caption{Non-cylindrical self-similar solutions of the IMCF given by Theorem \ref{thm-non-cyl} for $C=8,$ $k_1=1,$ and $k_2=0.$}\label{fig-non-cyl-2}
	\end{center}
\end{figure}

   \section{Cylindrical Homothetic Self-Similar Solution}\label{cyl}

   In this section, we classify the cylindrical nondegenerate ruled surfaces that are homothetic self-similar solutions of the IMCF in $\mathbb{L}^3.$ These surfaces have the parametrization
\[
X(s,t)=\gamma(s)+tw
\]
where $\gamma$ is a curve parametrized by the arc length and $w$ is a constant vector. Up to rigid motions of $\mathbb{L}^3,$ we can choose
\[
w=(1,0,1),\ w=(1,0,0), \ \mbox{or} \ w=(0,0,1).
\]
With this parametrization, we have $X_t=\gamma'(s)$ and $X_t=w,$ which implies $E=\interno{\gamma'}{\gamma'}=\delta\in\{-1,1\}$, $F=\interno{\gamma'}{w}$ and $G=\interno{w}{w}$. The case $w=(1,0,1)$ cannot happen for the IMCF since $w$ being lightlike and $X_{st}=X_{tt}=0$ imply 
\begin{equation}
        H=-\frac{1}{2}\frac{eG}{EG-F^2}=-\frac{1}{2}\frac{e\interno{w}{w}}{\interno{w}{w}-\interno{\gamma'}{w}^2}=0.
    \end{equation}
Thus we have that $w=(1,0,0)$ (i.e., $w$ is spacelike) or $w=(0,0,1)$ (that is, $w$ is timelike). We start analyzing the case when $w=(1,0,0).$

Let $M$ be a nondegenerate, cylindrical, smooth ruled surface $\mathbb{L}^3$ parametrized by 
\[
X(s,t)=\gamma(s)+t(1,0,0),
\]
where $\gamma$ is a curve parametrized by the arc length, lying in a plane orthogonal to $w$, i.e., 
$$
\gamma(s)=(0,x(s),y(s)),
$$
where
\begin{equation}\label{arc-l}
(x'(s))^2-(y'(s))^2=\delta\in\{-1,1\}
\end{equation}
is the arc length equation. Since
\[
\interno{N}{N}=\frac{\interno{X_s\times X_t}{X_s\times X_t}}{|EG-F^2|}=\frac{-\interno{X_s}{X_s}\interno{X_t}{X_t}+\interno{X_s}{X_t}^2}{|EG-F^2|}=-\delta,
\]
we have that the surface is spacelike if $\delta=1$ and timelike if $\delta=-1.$ This gives us the following scenario:
\begin{center}
\begin{tabular}{|c|c|}
\hline
   $\delta=-1$  & \begin{tabular}{|c|c|}
         $C>0$ & self-expander \\
         \hline
         $C<0$ & self-shrinker
     \end{tabular} \\
   \hline
   $\delta=1$  & \begin{tabular}{|c|c|}
        $C>0$ & self-shrinker \\
         \hline
         $C<0$ & self-expander
     \end{tabular}\\
     \hline
\end{tabular}
\end{center}
The first result of this section is

\begin{theorem}\label{thm-cyl-1}
Let $M$ be a nondegenerate, smooth, cylindrical, ruled surface in $\mathbb{L}^3,$ parametrized by  $X(s,t)=\gamma(s)+t w,$ where $w=(1,0,0)$ and $\gamma(s)=(0,x(s),y(s))$ is a curve parametrized by the arc length lying in a plane orthogonal to $w$. 
Let
\[
\left\{\aligned
f(s)&:=\delta\left(\dfrac{2}{C}-1\right)s^2+k,\\
t(s)&:=\pm\displaystyle\int\frac{\sqrt{4\delta f(s)+(f'(s))^2}}{2|f(s)|}ds=\pm\int\dfrac{\sqrt{2(2-C)s^2+\delta kC^2}}{|(2-C)s^2+\delta kC|}ds,\\
r(s)&:=\pm\sqrt{|f(s)|},
\endaligned\right.
\]
where $k\in\mathbb{R}.$ If $M$ is a homothetic self-similar solution of the inverse mean curvature flow, then
\[
\gamma(s)=\begin{cases}
    (0,r(s)\sinh(t(s)),r(s)\cosh(t(s)))&, \ \mbox{for} \ f(s)>0;\\
    (0,r(s)\cosh(t(s)),r(s)\sinh(t(s)))&, \ \mbox{for} \ f(s)<0,\\
\end{cases}
\]
(see Figure \ref{fig-cyl-1}).
\end{theorem}

\begin{proof}
Under this parametrization, we obtain $X_s=\gamma'(s)=(0,x'(s),y'(s))$ and $X_t=(1,0,0),$ which implies $E=\interno{\gamma'}{\gamma'}=\delta$, $F=0$ and $G=1$. Noticing that $X_s\times X_t=(0,y'(s),x'(s))$ and $(EG-F^2)^2=1$, Equation \eqref{princ} becomes
    \begin{equation}\label{trans-cyl}
  C(xy'-yx')(y'x''-x'y'')=2. 
		\end{equation}
  By taking derivatives in \eqref{arc-l}, we obtain 
  \[
  x'(s)x''(s)-y'(s)y''(s)=0.
  \] 
  Thus, we have the linear system of equations, in the unknowns  $x''(s)$ and $y''(s),$
  $$\begin{cases}
		x'x''-y'y''=0\\
		y'x''-x'y''=\displaystyle\frac{2}{C(xy'-yx')},
		\end{cases}
		$$
  whose solution is
  $$x''=\frac{-2\delta y'}{C(xy'-yx')}\hspace{1cm}\text{and}\hspace{1cm}y''=\frac{-2\delta x'}{C(xy'-yx')}.$$
  The above solutions are equivalent to
  	$$\frac{-xC\delta}{2}x''=\frac{x y'}{xy'-yx'}\hspace{1cm}\text{and}\hspace{1cm}\frac{yC\delta}{2}y''=\frac{-yx'}{xy'-yx'}.$$	
   By adding the last two equalities, we get the equation
   $$\frac{C\delta}{2}(yy''-xx'')=1$$
   i.e.,
   \begin{equation}\label{subs}
    yy''-xx''=\frac{2\delta}{C}.
\end{equation}
Defining the functions  
$$
u(s):=y(s)y'(s)=\dfrac{1}{2}(y(s)^2)'\hspace{0.5cm}\text{and}\hspace{0.5cm} v(s):=x(s)x'(s)=\dfrac{1}{2}(x(s)^2)',
$$
we have
\begin{equation}\label{uv}
u'(s)=y(s)y''(s)+(y'(s))^2\hspace{0.5cm}\text{and}\hspace{0.5cm}v'(s)=x(s)x''(s)+(x'(s))^2.
\end{equation}
By using \eqref{uv} and \eqref{arc-l} in \eqref{subs}, we obtain
\begin{equation}\label{uv-deriv}
  u'(s)-v'(s)=\delta\left(\frac{2}{C}-1\right).
\end{equation}
By translating the parameter $s$, the integration of \eqref{uv-deriv}, together with the definitions of $u$ and $v,$ imply
\begin{equation}
   u(s)-v(s)=\frac{1}{2}\left[(y(s)^2)'-(x(s)^2)'\right]=\delta\left(\frac{2}{C}-1\right)s.
\end{equation}
This is equivalent to
\begin{equation}\label{c2}
    y(s)^2-x(s)^2=\delta\left(\frac{2}{C}-1\right)s^2+k,
\end{equation}
where $k\in\mathbb{R}$ is a constant. Let 
\[
f(s):=\delta\left(\frac{2}{C}-1\right)s^2+k.
\]
Let us separate the analysis into two cases:
\begin{itemize}
    \item[(i)] If $f(s)>0,$ defining 
\begin{equation}\label{xy1}
    \begin{cases}
 x(s)=r(s)\sinh(t(s)),\\
 y(s)=r(s)\cosh(t(s)),\\
    \end{cases}
\end{equation}
for some functions $r(s)$ and $t(s),$ we obtain, by replacing \eqref{xy1} in \eqref{c2}, that
\begin{equation}\label{r1}
(r(s))^2=f(s)=|f(s)|.    
\end{equation}
Replacing \eqref{r1} and \eqref{xy1} in \eqref{arc-l}, we obtain
\[
-(r'(s))^2+(r(s))^2(t'(s))^2=\delta,
\]
which gives, after a multiplication by $4r^2,$ that
\[
-(f'(s))^2+4(f(s))^2(t'(s))^2=4\delta f(s).
\]
This gives the expression for $t(s)$ when $f(s)>0.$

\item[(ii)] If $f(s)<0,$ define
\begin{equation}\label{xy2}
    \begin{cases}
 x(s)=r(s)\cosh(t(s)),\\
 y(s)=r(s)\sinh(t(s)),\\
    \end{cases}
\end{equation}
where $r(s)$ and $t(s)$ are functions to be determined. Replacing \eqref{xy2} into \eqref{c2}, we obtain
\begin{equation}\label{r2}
(r(s))^2=-f(s)=|f(s)|.    
\end{equation}
Now, replacing \eqref{xy2} into \eqref{arc-l} gives
\[
-(r'(s))^2+(r(s))^2(t'(s))^2=-\delta,
\]
which gives, after multiplying the last equation by $4r^2$ and using the expression \eqref{r2}, that
\[
-(f'(s))^2+4(f(s))^2(t'(s))^2=4\delta f(s),
\]
which gives the expression for $t(s).$
\end{itemize}
\end{proof}

\begin{figure}[ht]
	\begin{center}
		\def\svgwidth{1\textwidth} 
		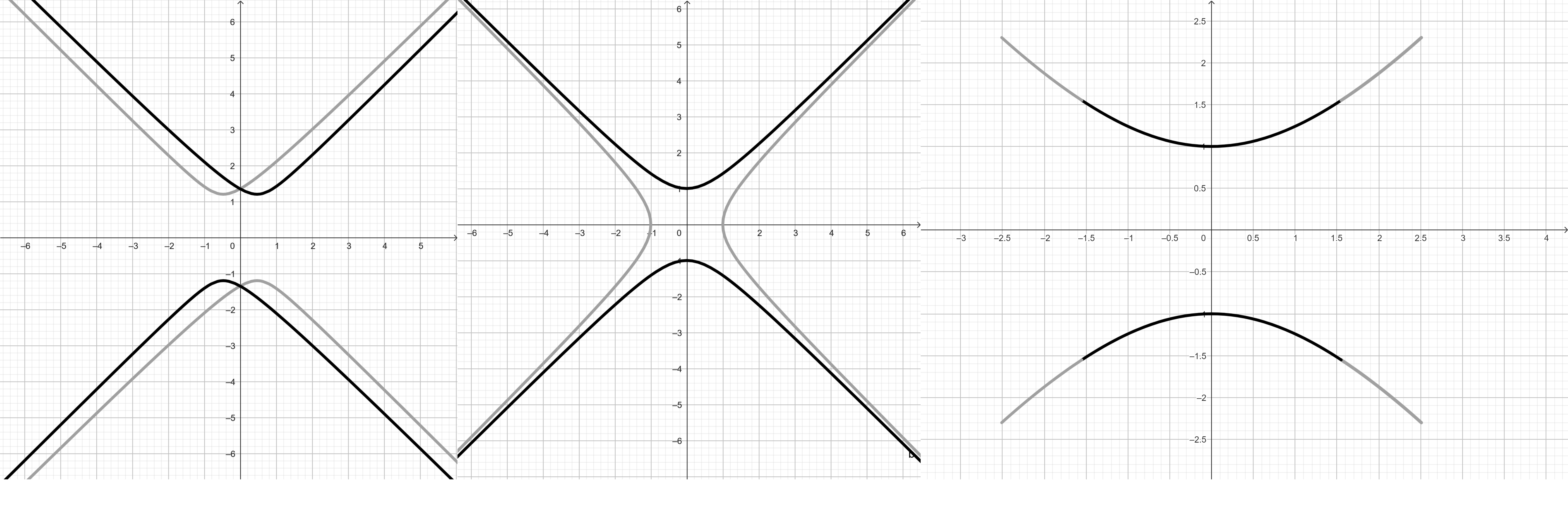
		\caption{Base curve of cylindrical self-similar solutions of the IMCF given by Theorem \ref{thm-cyl-1}. In each picture, we are drawing the four branches depending on the signal $+$ or $-$ in $t(s)$ and $r(s).$ The figures were drawn for (a) $C=k=\delta=1,$ (b) $C=2,$ $\delta=1,$ and $k=\pm1$ (in this case the curves are the hyperbola $x^2-y^2=\pm1),$ and $C=4,$ $\delta=k=1$. In this last case the curve is not complete.}\label{fig-cyl-1}
	\end{center}
\end{figure}
Now let us consider the case when $w=(0,0,1)$ and $M$ has the parametrization
\[
X(s,t)=\gamma(s)+t(0,0,1).
\]
Taking again the base curve $\gamma$ in a plane orthogonal do $w$ and parametrizing $\gamma$ by the arc length, we have
\[
\gamma(s)=(x(s),y(s),0)
\]
and
\begin{equation}\label{arc-l2}
    (x'(s))^2+(y'(s))^2=1.
\end{equation}
Since
\[
\interno{N}{N}=\frac{\interno{X_s\times X_t}{X_s\times X_t}}{|EG-F^2|}=\frac{-\interno{X_s}{X_s}\interno{X_t}{X_t}+\interno{X_s}{X_t}^2}{|EG-F^2|}=-1,
\]
we know that the surface is always spacelike. In this case, we have that the surface is a self-shrinker if $C<0$ and a self-expander if $C>0.$ The second result of this section is

\begin{theorem}\label{thm-cyl-2}
Let $M$ be a nondegenerate, smooth, cylindrical, ruled surface in $\mathbb{L}^3,$ parametrized by  $X(s,t)=\gamma(s)+t w,$ where $w=(0,0,1)$ and $\gamma(s)=(x(s),y(s),0)$ is a curve parametrized by the arc length lying in a plane orthogonal to $w$. If $M$ is a homothetic self-similar solution of the inverse mean curvature flow, then
\[
\gamma(s)=(r(s)\cos(t(s)),r(s)\sin(t(s)),0),
\]
where
\[
\begin{cases}
(r(s))^2&=\left(1-\frac{2}{C}\right)s^2+k>0,\\
t(s)&=\displaystyle\pm\int\frac{\sqrt{\frac{2}{C}\left(1-\frac{2}{C}\right)s^2+k}}{\left(1-\frac{2}{C}\right)s^2+k},
\end{cases}
\]
where $k\in\mathbb{R}$ is a constant (see Figure \ref{fig-cyl-2}).
\end{theorem}

\begin{figure}[ht]
	\begin{center}
		\def\svgwidth{1\textwidth} 
		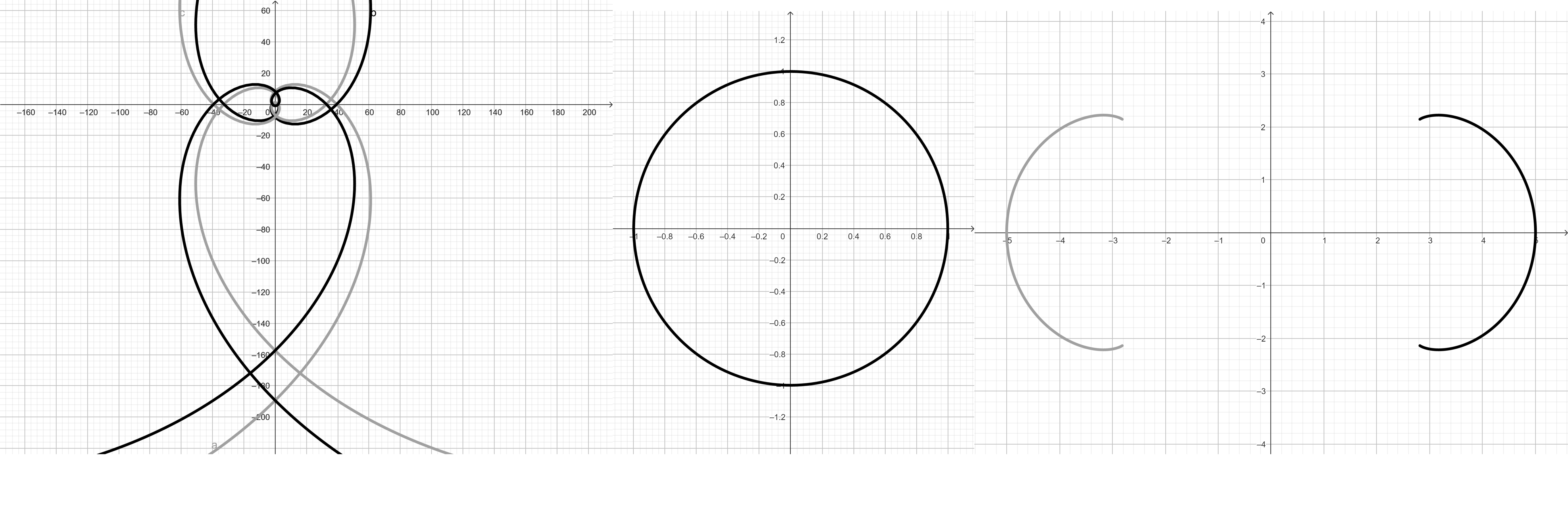
		\caption{Base curve of cylindrical self-similar solutions of the IMCF given by Theorem \ref{thm-cyl-1}. In each picture, we are drawing the four branches depending on the sinal $+$ or $-$ in $t(s)$ and $r(s).$ The figures were drawn for (a) $C=4$ and $k=1,$ (b) $C=2$ and $k=1$ (in this case the curve is the circle $x^2+y^2=1),$ and $C=1,$ $k=25$. In this last case, the curve is not complete.}\label{fig-cyl-2}
	\end{center}
\end{figure}

\begin{proof}
Since $X_s=(x'(s),y'(s),0)$ and $X_t=(0,0,1),$ we obtain $E=1,$ $F=0,$ and $G=-1.$ Thus, the self-similar solution equation \eqref{princ} becomes
  \begin{equation}\label{w-tempo}
  C(xy'-yx')(y'x''-x'y'')=-2.
  \end{equation}
Differentiating the arc length equation and using \eqref{w-tempo}, we obtain the system
\[
\begin{cases}
    x'x''+y'y''&=0\\
    y'x''-x'y''&=\dfrac{-2}{C(xy'-yx')},
\end{cases}
\]
whose solution is
\[
x''=\frac{-2y'}{C(xy'-yx')} \quad \mbox{and} \quad y''=\frac{2x'}{C(xy'-yx')}.
\]
This gives
\begin{equation}\label{subs2}
    xx''+yy''=-\frac{2}{C}.
\end{equation}
Define the functions  $$u(s):=y(s)y'(s)=\dfrac{1}{2}(y(s)^2)'\hspace{0.5cm}\text{and}\hspace{0.5cm} v(s):=x(s)x'(s)=\dfrac{1}{2}(x(s)^2)'.$$ 
By adding $u'$ and $v',$ and then replacing \eqref{subs2} in the result, we obtain
\begin{equation}\label{u+v}
  u'(s)+v'(s)=1-\frac{2}{C}.
\end{equation}
This gives, by integration and after translating the parameter $s,$ that
\begin{equation}
   u(s)+v(s)=\frac{1}{2}\left[(x(s)^2)'+(y(s)^2)'\right]=\left(1-\frac{2}{C}\right)s,
\end{equation}
i.e.,
\begin{equation}\label{xy-sq}
    x(s)^2+y(s)^2=\left(1-\frac{2}{C}\right)s^2+k=:f(s),
\end{equation}
where $k\in\mathbb{R}$ is a constant. Notice that $f(s)$ defined in \eqref{xy-sq} is nonnegative. Let
\begin{equation}\label{xy3}
\begin{cases}
    x(s)&=r(s)\cos(t(s)),\\
    y(s)&=r(s)\sin(t(s)),\\
\end{cases}
\end{equation}
where $r(s)$ and $t(s)$ are functions to be determined. Replacing \eqref{xy3} into \eqref{xy-sq} gives 
\begin{equation}\label{r3}
r(s)^2=f(s).    
\end{equation}
Replacing \eqref{r3} and \eqref{xy3} into \eqref{arc-l2}, we obtain
\[
(r'(s))^2+(r(s))^2(t'(s))^2=1,
\]
which is equivalent (after multiplying by $4r^2$) to
\[
(f'(s))^2+4(f(s))^2(t'(s))^2=4f(s).
\]
This gives, after integration in $s$, the expression for $t(s).$
\end{proof}

\end{document}

%% file: expander-lightlike.pdf_tex
\begingroup%
  \makeatletter%
  \providecommand\color[2][]{%
    \errmessage{(Inkscape) Color is used for the text in Inkscape, but the package 'color.sty' is not loaded}%
    \renewcommand\color[2][]{}%
  }%
  \providecommand\transparent[1]{%
    \errmessage{(Inkscape) Transparency is used (non-zero) for the text in Inkscape, but the package 'transparent.sty' is not loaded}%
    \renewcommand\transparent[1]{}%
  }%
  \providecommand\rotatebox[2]{#2}%
  \newcommand*\fsize{\dimexpr\f@size pt\relax}%
  \newcommand*\lineheight[1]{\fontsize{\fsize}{#1\fsize}\selectfont}%
  \ifx\svgwidth\undefined%
    \setlength{\unitlength}{994.4999827bp}%
    \ifx\svgscale\undefined%
      \relax%
    \else%
      \setlength{\unitlength}{\unitlength * \real{\svgscale}}%
    \fi%
  \else%
    \setlength{\unitlength}{\svgwidth}%
  \fi%
  \global\let\svgwidth\undefined%
  \global\let\svgscale\undefined%
  \makeatother%
  \begin{picture}(1,0.7533937)%
    \lineheight{1}%
    \setlength\tabcolsep{0pt}%
    \put(0,0){\includegraphics[width=\unitlength,page=1]{expander-lightlike.pdf}}%
  \end{picture}%
\endgroup%

%% file: lightlike-2.pdf_tex
\begingroup%
  \makeatletter%
  \providecommand\color[2][]{%
    \errmessage{(Inkscape) Color is used for the text in Inkscape, but the package 'color.sty' is not loaded}%
    \renewcommand\color[2][]{}%
  }%
  \providecommand\transparent[1]{%
    \errmessage{(Inkscape) Transparency is used (non-zero) for the text in Inkscape, but the package 'transparent.sty' is not loaded}%
    \renewcommand\transparent[1]{}%
  }%
  \providecommand\rotatebox[2]{#2}%
  \newcommand*\fsize{\dimexpr\f@size pt\relax}%
  \newcommand*\lineheight[1]{\fontsize{\fsize}{#1\fsize}\selectfont}%
  \ifx\svgwidth\undefined%
    \setlength{\unitlength}{890.25002307bp}%
    \ifx\svgscale\undefined%
      \relax%
    \else%
      \setlength{\unitlength}{\unitlength * \real{\svgscale}}%
    \fi%
  \else%
    \setlength{\unitlength}{\svgwidth}%
  \fi%
  \global\let\svgwidth\undefined%
  \global\let\svgscale\undefined%
  \makeatother%
  \begin{picture}(1,0.79443977)%
    \lineheight{1}%
    \setlength\tabcolsep{0pt}%
    \put(0,0){\includegraphics[width=\unitlength,page=1]{lightlike-2.pdf}}%
  \end{picture}%
\endgroup%

%% file: SS-non-cyl-1.pdf_tex
\begingroup%
  \makeatletter%
  \providecommand\color[2][]{%
    \errmessage{(Inkscape) Color is used for the text in Inkscape, but the package 'color.sty' is not loaded}%
    \renewcommand\color[2][]{}%
  }%
  \providecommand\transparent[1]{%
    \errmessage{(Inkscape) Transparency is used (non-zero) for the text in Inkscape, but the package 'transparent.sty' is not loaded}%
    \renewcommand\transparent[1]{}%
  }%
  \providecommand\rotatebox[2]{#2}%
  \newcommand*\fsize{\dimexpr\f@size pt\relax}%
  \newcommand*\lineheight[1]{\fontsize{\fsize}{#1\fsize}\selectfont}%
  \ifx\svgwidth\undefined%
    \setlength{\unitlength}{978.75bp}%
    \ifx\svgscale\undefined%
      \relax%
    \else%
      \setlength{\unitlength}{\unitlength * \real{\svgscale}}%
    \fi%
  \else%
    \setlength{\unitlength}{\svgwidth}%
  \fi%
  \global\let\svgwidth\undefined%
  \global\let\svgscale\undefined%
  \makeatother%
  \begin{picture}(1,0.74789273)%
    \lineheight{1}%
    \setlength\tabcolsep{0pt}%
    \put(0,0){\includegraphics[width=\unitlength,page=1]{SS-non-cyl-1.pdf}}%
  \end{picture}%
\endgroup%

%% file: SS-non-cyl-2.pdf_tex
\begingroup%
  \makeatletter%
  \providecommand\color[2][]{%
    \errmessage{(Inkscape) Color is used for the text in Inkscape, but the package 'color.sty' is not loaded}%
    \renewcommand\color[2][]{}%
  }%
  \providecommand\transparent[1]{%
    \errmessage{(Inkscape) Transparency is used (non-zero) for the text in Inkscape, but the package 'transparent.sty' is not loaded}%
    \renewcommand\transparent[1]{}%
  }%
  \providecommand\rotatebox[2]{#2}%
  \newcommand*\fsize{\dimexpr\f@size pt\relax}%
  \newcommand*\lineheight[1]{\fontsize{\fsize}{#1\fsize}\selectfont}%
  \ifx\svgwidth\undefined%
    \setlength{\unitlength}{960.7499827bp}%
    \ifx\svgscale\undefined%
      \relax%
    \else%
      \setlength{\unitlength}{\unitlength * \real{\svgscale}}%
    \fi%
  \else%
    \setlength{\unitlength}{\svgwidth}%
  \fi%
  \global\let\svgwidth\undefined%
  \global\let\svgscale\undefined%
  \makeatother%
  \begin{picture}(1,0.76190479)%
    \lineheight{1}%
    \setlength\tabcolsep{0pt}%
    \put(0,0){\includegraphics[width=\unitlength,page=1]{SS-non-cyl-2.pdf}}%
  \end{picture}%
\endgroup%

%% file: SS-cyl-1.pdf_tex
\begingroup%
  \makeatletter%
  \providecommand\color[2][]{%
    \errmessage{(Inkscape) Color is used for the text in Inkscape, but the package 'color.sty' is not loaded}%
    \renewcommand\color[2][]{}%
  }%
  \providecommand\transparent[1]{%
    \errmessage{(Inkscape) Transparency is used (non-zero) for the text in Inkscape, but the package 'transparent.sty' is not loaded}%
    \renewcommand\transparent[1]{}%
  }%
  \providecommand\rotatebox[2]{#2}%
  \newcommand*\fsize{\dimexpr\f@size pt\relax}%
  \newcommand*\lineheight[1]{\fontsize{\fsize}{#1\fsize}\selectfont}%
  \ifx\svgwidth\undefined%
    \setlength{\unitlength}{4905bp}%
    \ifx\svgscale\undefined%
      \relax%
    \else%
      \setlength{\unitlength}{\unitlength * \real{\svgscale}}%
    \fi%
  \else%
    \setlength{\unitlength}{\svgwidth}%
  \fi%
  \global\let\svgwidth\undefined%
  \global\let\svgscale\undefined%
  \makeatother%
  \begin{picture}(1,0.33654945)%
    \lineheight{1}%
    \setlength\tabcolsep{0pt}%
    \put(0,0){\includegraphics[width=\unitlength,page=1]{SS-cyl-1.pdf}}%
    \put(0.14340025,0.00277289){\color[rgb]{0,0,0}\makebox(0,0)[lt]{\lineheight{1.25}\smash{\begin{tabular}[t]{l}$(a)$\end{tabular}}}}%
    \put(0.76551988,0.0008989){\color[rgb]{0,0,0}\makebox(0,0)[lt]{\lineheight{1.25}\smash{\begin{tabular}[t]{l}$(c)$\end{tabular}}}}%
    \put(0.42933604,0.0027524){\color[rgb]{0,0,0}\makebox(0,0)[lt]{\lineheight{1.25}\smash{\begin{tabular}[t]{l}$(b)$\end{tabular}}}}%
  \end{picture}%
\endgroup%

%% file: SS-cyl-2.pdf_tex
\begingroup%
  \makeatletter%
  \providecommand\color[2][]{%
    \errmessage{(Inkscape) Color is used for the text in Inkscape, but the package 'color.sty' is not loaded}%
    \renewcommand\color[2][]{}%
  }%
  \providecommand\transparent[1]{%
    \errmessage{(Inkscape) Transparency is used (non-zero) for the text in Inkscape, but the package 'transparent.sty' is not loaded}%
    \renewcommand\transparent[1]{}%
  }%
  \providecommand\rotatebox[2]{#2}%
  \newcommand*\fsize{\dimexpr\f@size pt\relax}%
  \newcommand*\lineheight[1]{\fontsize{\fsize}{#1\fsize}\selectfont}%
  \ifx\svgwidth\undefined%
    \setlength{\unitlength}{5170.49993079bp}%
    \ifx\svgscale\undefined%
      \relax%
    \else%
      \setlength{\unitlength}{\unitlength * \real{\svgscale}}%
    \fi%
  \else%
    \setlength{\unitlength}{\svgwidth}%
  \fi%
  \global\let\svgwidth\undefined%
  \global\let\svgscale\undefined%
  \makeatother%
  \begin{picture}(1,0.32285603)%
    \lineheight{1}%
    \setlength\tabcolsep{0pt}%
    \put(0,0){\includegraphics[width=\unitlength,page=1]{SS-cyl-2.pdf}}%
    \put(0.13597644,0.00284252){\color[rgb]{0,0,0}\makebox(0,0)[lt]{\lineheight{1.25}\smash{\begin{tabular}[t]{l}$(a)$\end{tabular}}}}%
    \put(0.47974248,0.00394991){\color[rgb]{0,0,0}\makebox(0,0)[lt]{\lineheight{1.25}\smash{\begin{tabular}[t]{l}$(b)$\end{tabular}}}}%
    \put(0.78891015,0.00512212){\color[rgb]{0,0,0}\makebox(0,0)[lt]{\lineheight{1.25}\smash{\begin{tabular}[t]{l}$(c)$\end{tabular}}}}%
  \end{picture}%
\endgroup%